\documentclass[11pt]{article}

\usepackage[latin1]{inputenc}
\usepackage[english]{babel}

\usepackage{amsmath,amsthm,amssymb}
\usepackage{epsfig}
\usepackage{graphics}
\usepackage{graphicx}
\usepackage{booktabs}
\usepackage{ctable}
\usepackage{subfigure}
\usepackage{float}
\usepackage{layout}
\usepackage{fancyhdr}
\usepackage{colortbl}  
\usepackage{color}     
\usepackage{booktabs}  
\usepackage{longtable} 
\usepackage{lscape}    
\usepackage{rotating}  
\usepackage{multirow}  
\usepackage{makeidx}   
\usepackage{varioref}
\usepackage{xr}
\usepackage{calc}
\usepackage{eurosym}   

\input amssym.def
\input amssym.tex

\makeindex 

\def \frac#1#2{{#1\over #2}}
\def\1{{\bf 1}}
\def\0{{\bf 0}}
\def\2{{\bf 2}}

\def\x{{\bf x}}

\def\v{{\bf v}}
\def\w{{\bf w}}
\def\b{{\bf b}}

\def\cfrac#1#2{{#1\over #2}}

\newtheorem{defi}{Definition}[section]
\newtheorem{risu}{Result}[section]
\newtheorem{teo}{Theorem}[section]
\newtheorem{prop}{Proposition}[section]

\newtheorem{rem}{Remark}[section]

\newcommand{\be}{\begin{enumerate}}
\newcommand{\ee}{\end{enumerate}}
\newcommand{\bi}{\begin{itemize}}
\newcommand{\ei}{\end{itemize}}
\newcommand{\beq}{\begin{equation}}
\newcommand{\eeq}{\end{equation}}

\numberwithin{equation}{section} 

\begin{document}

\baselineskip=14 pt

\title{Twin relationships in Parsimonious Games: some results\footnote{A very preliminary version of the paper has been presented at the 2013 Workshop of the Central European Program in Economic Theory, which took place in Udine (20-21 June) and may be found in CEPET working papers \cite{PrePla13}.}}

\author{Flavio Pressacco\textsuperscript{a}, Giacomo Plazzotta\textsuperscript{b} and Laura Ziani\textsuperscript{c}}

\date{}

\maketitle

\noindent\textsuperscript{a} Dept. of Economics and Statistics D.I.E.S., Udine University, Italy\\ 
\textsuperscript{b} Imperial College London, UK\\
\textsuperscript{c} Dept. of Economics and Statistics D.I.E.S., Udine University, Italy 

\begin{abstract}
In a vintage paper concerning Parsimonious games, a subset of
constant sum homogeneous weighted majority games, Isbell introduced a twin
relationship based on transposition properties of the incidence
matrices upon minimal winning coalitions of such games. A
careful investigation of such properties allowed the discovery
of some results on twin games presented in this paper. In
detail we show that a) twin games have the same minimal winning
quota and b) each Parsimonious game admits a unique balanced
lottery on minimal winning coalitions, whose probabilities are
given by the individual weights of its twin game.

\end{abstract}

\renewcommand{\abstractname}{Keywords}
\begin{abstract}
Homogeneous weighted majority games,
incidence matrices, twin relationships, minimal winning quota,
balanced lottery.
\end{abstract}

\renewcommand{\abstractname}{Acknowledgements}
\begin{abstract}
 We acknowledge financial support of MedioCredito Friuli Venezia Friuli through the ``Bonaldo Stringher'' Laboratory of Finance, Department of Finance, University of Udine.
\end{abstract}

\section{Introduction}

In this paper we present some results concerning twin
relationships in Parsimonious games (henceforth $P$ games). $P$
games are the subset of constant sum homogeneous weighted
majority games characterized by the parsimony property to have,
for any given number $n$ $(>3)$\footnote{We will consider here $P$
games with $n>3$. Indeed, for $n=3$ there is a unique $P$ game
in which all players share the same weight, while all $P$ games
with $n>3$ have at least two types of players. See \cite{Isb56}, p. 185.} of non dummy
players in the game, the smallest number, i.e. exactly $n$, of
minimal winning coalitions. $P$ games have been defined and
studied by Isbell in a vintage paper (\cite{Isb56},
1956) where the special properties of this class of games are
described. Among other things Isbell introduced a twin
relationship on $P$ games, but without deepening the point,
which at the best of our knowledge did not receive any further
attention in the relevant literature.

To understand the twin relationship it is convenient to recall
that its premise is the existence of a general rule that drives
the one to one correspondence that obviously exists between the
set of players and the set of minimal winning coalitions of any
$P$ game. In turn this correspondence comes out from the
following idea: keeping account of the minimal homogeneous
representation of a given $P$ game, divide the players in $h$
groups $(2\leq h\leq n-2)$ so as all members of the same group
(of the same type) share the same individual weight. Order
players $j=1,\ldots,n$ and types $t=1,\ldots,h$ according to a
non decreasing (for players) or a strictly increasing (for
types) weight convention. In particular, players of the group with
minimum weight (type 1) may be called ``peones'', those of type
$h-1$ ``vice-top'' and the player (of group $h$) with greatest
weight ``top player''. In any $P$ game there are lower bounds
on the number of peones and of vice top and a binding constraint on the top class: indeed
there is just one top player.

Let us shortly call odd (respectively even) players, those
whose type is odd (even). After that, the one to one
correspondence is described by the following rules: to
any non top player there is associated the minimal
winning coalition made by that player and all players of
alternative parity and greater weight, whereas the coalition
made by all odd players (which may include or not the top
player) is associated to the top player. Note that, in this
way, also the set of minimal winning coalitions is divided in
$h$ groups and, matching the type order of players and
associated coalitions, that the numerousness $x_t$ of type $t$
players is equal to the one of type $t$ minimal winning
coalitions. All these results imply a special structure of the
incidence matrix upon minimal winning coalitions of a $P$ game, i.e. the square $n$ dimensional
binary matrix $A$ whose elements $a_{ij}$ are 1 if column
player $j$ belongs to the row minimal winning coalition $i$ (or
$S_i$), and 0 otherwise. In particular the square submatrix
$M$, obtained from $A$ by deleting the last row and column, is
block diagonal upper triangular. The diagonal blocks of $M$ are
$x_t$ dimensional diagonal square matrices, rectangular blocks
over the diagonal alternate matrices with all elements equal to
one to null matrices, while by definition all the entries under the diagonal blocks are zero.

This special structure inspired Isbell in recognizing that the
transposed $A^{T}$ of the incidence matrix $A$ of any ``primal''
game $G$ should still be the incidence matrix of a $P$ game
${\overline{G}}$ to be called twin (dual in Isbell terminology)
of $G$. In other words a couple of $P$ games are twins if and
only if each incidence matrix of the couple is obtained by
transposition of the incidence matrix of the twin\footnote{As we shall see later (Prop. \ref{prop:333}) an alternative fully equivalent definition of twins could be based on the symmetry of the free type representations of twin games.}.

Our contribution in this paper is embedded in a couple of
theorems regarding twins. The first result says that twins have
the same minimal winning quota, and is a straightforward
corollary of a theorem regarding the determinant of the
incidence matrix of any $P$ game. The theorem states that the
absolute value of the determinant is the minimal winning quota
of the game.

The second theorem regards the connection between the balanced
lottery on minimal winning coalitions of a $P$ game and the
individual weights of its twin. To understand the point suppose
that the following three-stage mechanism is adopted to fix a
fair and stable result for a $P$ game, avoiding lengthy and
may be unsatisfying negotiations between players. In the first
stage the probabilities driving a lottery on the set of minimal
winning coalitions are fixed. In the second stage the lottery
mechanism makes the choice of one minimal winning coalition and
in the third the (normalized to one) global reward of the game
is divided within players of the chosen minimal winning
coalition. In order to grant ex post fairness, in this stage
individual rewards are proportional to individual weights.

We define balanced\footnote{Generally speaking, a given set of
$m$ coalitions $\mathbf{S}=(S_1,\ldots,S_i,\ldots,S_m$) is
balanced if there exists a positive $m-$dimensional vector
$\mathbf{d}$ such that for any (non dummy) player $j$ it is
$\sum_{i:j \in S_i}d_i=1$. Then, except for a normalization, a
balanced lottery on a $P$ game has the property to balance the
set of minimal winning coalitions of the game.} a lottery that
gives to all players the same (ex ante) probability to be a
member of the chosen minimal winning coalition. Clearly a
balanced lottery adds ex ante to ex post fairness: indeed it is
easy to check that (given the ex post mechanism of reward) only
under a balanced lottery the expected gain of each player is
exactly his (normalized) individual weight. A question
immediately arises: any $P$ game admits (may be just one)
balanced lotteries? The answer is positive: our second
fundamental theorem shows that for any $P$ game there exists
just one balanced lottery which is given by the normalized and
properly reordered set of individual weights of its twin.
Precisely the balanced probabilities of type $t$ coalitions in
the game ${G}$ are, for any $t$ but the top one, the normalized
weights of type $h-t$ players in the twin ${\overline{G}}$,
while the probability of the coalition associated to the top
player is just the normalized weight of the top player in
${\overline{G}}$. The transposition properties of incidence
matrices play a decisive role also in the proof of this result.

The plan of the paper is as follows. In section 2 a
short recall of the basics of constant sum homogeneous weighted
majority games is offered; section 3, divided in three
subsections (one to one correspondence between players and
minimal winning coalitions; the special structure of the
incidence matrix of a $P$ game; the transposition approach to
the twin relationships in $P$ games) resumes fundamental
results on $P$ games\footnote{Such results may be found in Isbell paper or are straightforward consequences of his work.}. Section 4 (minimal
winning quota in twin games) and 5 (balanced lotteries
and twin games) give proofs of our main theorems on twin games.
Some examples are offered in section 6. Conclusions
follow in the final section 7.

\section{Basics on homogeneous weighted majority games}

Let us recall some well known basic definitions.

As usual $N=\{1,\ldots,n\}$ denotes the set of all (non dummy)
players.

A coalition $S$ is a subset of the set $N$. A game $G$ in
coalitional function form is defined by the coalitional
function of the game, that is the real function $v:
\mathcal{P}(N) \rightarrow \mathbb{R}$.

A game in coalitional function form is simple if its $v$
function has values in $\{0,1\}$. A coalition $S$ is winning if
$v(S)=1$, losing if $v(S)=0$.

A simple game is constant sum if, for any $S \in
\mathcal{P}(N)$, $v(S)+v(\widetilde{S})=1$.

A coalition $S$ is said to be minimal winning if $v(S)=1$ and,
for any $T\varsubsetneq S$, $v(T)=0$. The set of minimal
winning coalitions is denoted by $WM$. A player $j$ who does
not belong to any minimal winning coalition is said dummy; at
the other extreme, it is a dictator if it is $v(j)=1$. We will
consider games free of dictator and dummies.

A simple, weighted majority game is described by a
representation $(q;\w)$ where $\w$ is a vector
$(w_1,w_2,\ldots,w_n)$ of positive weights,
$q>\cfrac{1}{2}\cdot w(N)=\cfrac{1}{2}\cdot\sum_{j\in N}w_j$ is
the winning quota and $v(S)=1\Leftrightarrow w(S)=\sum_{j\in
S}w_j\geq q$.

A representation $(q;\w)$ of a weighted majority game is
homogeneous if $w(S)=q$ for any $S$ of $WM$.

Homogeneous weighted majority games are games for which (at
least) a homogeneous representation exists.

Consider now the class of all simple, constant sum $n$ person
homogeneous weighted majority games\footnote{At the origins of
game theory, homogeneous weighted majority games (h.w.m.g.)
have been introduced in \cite{VNM47} by Von Neumann-Morgenstern
and have been studied mainly under the constant sum condition.
Subsequent treatments in the absence of the constant sum
condition (with deadlocks) may be found e.g. in \cite{Ost87} by
Ostmann, who gave the proof that any h.w.m.g. (including non
constant sum ones) has a unique minimal homogeneous
representation, and in \cite{Ros87}. Generally speaking, the
homogeneous minimal representation is to be thought in a
broader sense but hereafter the restrictive application
concerning the constant sum case is used.}.
\begin{prop}
All games of such a class admit a minimal homogeneous
representation, that is a (homogeneous) representation $(q;\w)$
such that all weights are integers, there are players with
(minimum) weight 1 and $q=\cfrac{1+w(N)}{2}$, which implies for
any $S \in WM$, $w(S)-w(\widetilde{S})=1$.
\end{prop}

Hereafter we will suppose that in such a representation the
vector $\w$ is ordered according to the convention $w_1=1,
w_j\leq w_{j+1}$ for any $j$.

\begin{prop}
The cardinality of the $WM$ set of our class may be either
greater or equal (but not lower) than $n$. See \cite{Isb56}, p.
185.
\end{prop}

\begin{defi}
We call Parsimonious games (hereafter $P$ games) the subset of
constant sum homogeneous weighted majority games characterized
by the parsimony property to have for any given number $n$ of
non dummy players in the game the smallest number, i.e. exactly
$n$, of minimal winning coalitions.
\end{defi}

As said before, general properties of $P$ games have been studied by Isbell. A recall of such
properties is given in the next section divided in three
subsections devoted respectively to:
\begin{itemize}
\item[a] The one to one correspondence between players and
    minimal winning coalitions
\item[b] The incidence matrix upon winning minimal coalitions of a $P$ game
\item[c] The transposition approach to the twin relationship
    in $P$ games
\end{itemize}

\section{Main results on Parsimonious games}
\subsection{One to one correspondence between players and minimal winning
coalitions}

Let us recall that in any $n$ person $P$ game there are $h$
$(2\leq h\leq n-2)$ types; players of type $t$ share, in the
minimal homogeneous representation, a type weight $w_t$ and the
ordering of types satisfies $w_t<w_{t+1}$. As said in the
introduction, odd (even) players are those of type $t$ odd
(even).
Let us denote by $x_t$ the numerousness of players of type $t$ and give the following:
\begin{defi}\label{def:31}
Let $G$ be a $P$ game with $h$ types; the type representation of $G$ is the vector $\x=(x_1,\ldots,x_t,\ldots,x_h)$, while the free type representation $_f\x=(x_1,\ldots,x_t,\ldots,x_{h-1})$ is the one obtained by deleting its last component $x_h$.
\end{defi}

The main Isbell result (\cite{Isb56}, p. 185, penultimate indent) was that
\begin{prop}
In all $P$ games the one to one correspondence between players
and minimal winning coalitions is described by a unique general
rule. Precisely, to any non top player (of type $t<h$) there
corresponds the coalition made by that player and all players
of alternative parity and greater weight; to the top player the
coalition made by all odd players (which may or not include the
top player).
\end{prop}
Keeping account that all minimal winning coalitions share the
same minimal winning quota and hence the same sum of individual
weights, the following consequences of the general rule are
straightforward for the sequence of type weights:
\begin{subequations}\label{Eqn:recursive}
\begin{align}
w_1&=1\quad \text{initial condition}\label{Eqn1:first}\\
w_2&=x_1\cdot w_1=x_1\label{Eqn1:second}\\
w_t&=x_{t-1}\cdot w_{t-1}+ w_{t-2}, \quad t=3,\ldots, h-1\label{Eqn1:third}\\
w_h&=(x_{h-1}-1)\cdot w_{h-1}+w_{h-2}\label{Eqn1:forth}
\end{align}
\end{subequations}

Hence \begin{prop} In any $P$ game with $h$ types the weights
of all types are unequivocally given as the recursive implicit function
\eqref{Eqn:recursive} of the free type representation
$_f\x=(x_1,\ldots,x_t,\ldots,x_{h-1})$.
\end{prop}

Moreover \begin{prop} To preserve strict monotony of the
sequence of type weights, the type representation  needs to satisfy,
in addition to the binding constraint $x_h=1$, the lower
bounds: $x_1>1$ (otherwise $w_2=w_1=1$) and $x_{h-1}>1$
(otherwise $w_h=w_{h-2}$). No other constraints hold (\cite{Isb56}, p. 185).
\end{prop}

In the next section we will exploit also the following
relations (\cite{Isb56}, p. 186, first indent) between type weights of the players:
\begin{subequations}
\begin{align}
\text{for $t$ odd: } \quad w_t&=1+\sum_{s \hspace{0.1cm}\text{even}<t}x_s\cdot w_s \label{Eqn2:first}\\
\text{for $t$ even: } \quad w_t&=\sum_{s \hspace{0.1cm} \text{odd}<t}x_s\cdot w_s\label{Eqn2:second}
\end{align}
\end{subequations}

Formula \eqref{Eqn2:first} comes from the minimal winning
character of the coalitions made by one of the peones and all
even players and, respectively, by one player of type $t$ (odd)
and all even players of greater weight; formula
\eqref{Eqn2:second} from the minimal winning character of the
coalitions made by all odd players and, respectively, by one
player of type $t$ even and all odd players of greater weight.
Hence
\begin{prop}
\begin{equation}
q=1+\sum_{s \hspace{0.1cm}\text{even}}x_s\cdot w_s=\sum_{s \hspace{0.1cm} \text{odd}}x_s\cdot w_s
\end{equation}
\begin{equation}
w(N)=1+2\cdot \sum_{s \hspace{0.1cm}\text{even}}x_s\cdot w_s=2\cdot\sum_{s \hspace{0.1cm} \text{odd}}x_s\cdot w_s-1
\end{equation}
\end{prop}

\subsection{The structure of the incidence matrix of a $P$ game}\label{sect:32}

The incidence matrix $A$ upon minimal winning coalitions of a $P$ game is the square binary
$n\times n$ matrix obtained through an association of players
to columns (in order of non decreasing weights), and of minimal
winning coalitions to rows (in the order induced by the one to
one correspondence explained in the previous subsection) and
putting $a_{ij}=1$ if player $j$ is a member of the minimal
winning coalition $i$ (or $S_i$ associated to player $i$), and
$a_{ij}=0$ otherwise.

The square submatrix $M$ obtained from $A$ by deleting the last
row and the last column turns out to be a matrix with a special
structure.

Denoting by $B_{r,c}$, with $r,c=1,\ldots,h-1$ the block
(dimension $x_r\times x_c$) associated to coalitions of type
$r$ and players of type $c$, it turns out that:
\begin{prop}\label{prop:321}
Diagonal blocks $(r=c)$ are square identity matrices;
\end{prop}
\begin{prop}\label{prop:322}
Those with $c<r<h$, i.e. under the diagonal, are rectangular
null matrices;
\end{prop}
\begin{prop}\label{prop:323}
Those with $r<c<h$ are rectangular with elements identically
equal to 1(0) if $(r+c)$ is odd (even).
\end{prop}

Moreover keeping in consideration also the last row and column
of $A$ there are other blocks with one or both $r,c$ equal to
$h$:
\begin{prop}\label{prop:324}
For $c=h$ the elements of the blocks are still identically
equal to 1 (0) if $(r+c)$ is odd (even);
\end{prop}

\begin{prop}\label{prop:325}
For $r=h$ the elements are identically equal to 1
(0) for $c$ odd (even).
\end{prop}

As we shall see in next sections, the incidence matrix is a
helpful tool to study a $P$ game and in particular the
connections between twin $P$ games.

\subsection{Twin relationships in $P$ games}

The special structure of the incidence matrices of $P$ games
inspired Isbell (\cite{Isb56}, p. 185, last indent) in recognizing an important property embedded
in the following:
\begin{prop}\label{prop:331}
The transposed $A^T$ of the incidence matrix of any (primal)
$P$ game $G$ is still the incidence matrix of a $P$ game
$\overline{G}$, to be seen as the twin (dual in Isbell
terminology) of $G$.
\end{prop}

\begin{rem}
It is important to underline that in $A^T$ the association of
all non top players to columns and of the corresponding minimal
winning coalitions to rows holds just in the reverse order
(from the last but most powerful in the game to the least powerful) than
the standard one (which goes from the least powerful to the
last but most powerful). In more detail in $A^T$ for any
$t=1,\ldots,h-1$ and $r<h$ the block indexed $B_{r,t}$ is
associated to the set of players of type $h-t$ and to
coalitions of type $h-r$ in $\overline{G}$; only the last
column still remains associated to the top player (for $r<h$
$B_{r,h}$ concerns coalitions of type $h-r$ and the top player
in $\overline{G}$) and the last row to the coalition
corresponding to the top player (for $r<h$ $B_{h,r}$ concerns
players of type $h-r$ and the coalition of type $h$ in
$\overline{G}$).
\end{rem}

\begin{rem}
For all $t<h$ the number of players (and of the associated
coalitions) of type $t$ in $G$ is the number of players (and of
the associated coalitions) of type $h-t$ in $\overline{G}$.
\end{rem}

Hence, more formally:
\begin{prop}\label{prop:332}
$M^T$, the transposed matrix of the submatrix $M$ (of $A$),
should be the incidence submatrix $\overline{M}$ of another $P$
game $\overline{G}$, but now with $\overline{M}$ block diagonal
lower triangular with ordering of coalitions and players
inverted (i.e. non increasing) with respect to that of $M$. To
obtain $\overline{A}$ complete the transposition of the last
row and column.
\end{prop}

If we now wish to recover the standard ordering in
$\overline{A}$, we should transform $A^T$ in the following way
to obtain a modified transposed $A^{\tau}$: at first rewrite
all rows of $A^T$ inverting the order of all but the last entry
in each row, then rewrite all columns inverting the order of
all but the last entry in each column. It turns out that:
\begin{risu}
\begin{equation}\label{Eqn:ris3.1}
\begin{split}
a_{i,j}^{T}&=a_{n-i,n-j}^{\tau} \hspace{0.2cm} \text{for any $i,j<n$}\\
a_{i,j}^{T}&=a_{n-i,n}^{\tau} \hspace{0.4cm} \text{for any $i<n$}\\
a_{n,j}^{T}&=a_{n,n-j}^{\tau} \hspace{0.4cm} \text{for any $j<n$}\\
a_{n,n}^{T}&=a_{n,n}^{\tau} 
\end{split}
\end{equation}
\end{risu}

After these modifications $A^{\tau}$ turns out to be the
incidence matrix $\overline{A}$ of $\overline{G}$ coherent with
the standard ordering convention of players and coalitions. To
understand the point see also the examples in section \ref{Sect:examples}.

We signal that an alternative, easier and more immediate
understanding of the connection between a couple of twin P
games, is given\footnote{As suggested by a rather cryptic
sentence in \cite{Isb56}, last row, p. 185, on the point see
also \cite{PPZ1}, sect. 3, p. 7.} through their free type
representations and is resumed by:

\begin{prop}\label{prop:333}
For any $t=1,\ldots,h-1$, $\overline{x_t}=x_{h-t}$.
\end{prop}

The proposition says that a couple of twins share the same
value of $n$ and $h$ and that their free type representation vectors
are each other symmetric.

\section{The minimal winning quota in twin $P$ games}

In this section we will show that:

\begin{teo}\label{teo:41}
Let $G$, $\overline{G}$ be any couple of twin $P$ games and $q$
and $\overline{q}$ their minimal winning quotas; then $q=\overline{q}$.
\end{teo}
Theorem \ref{teo:41} is a straightforward corollary of the
following theorem linking the minimal winning quota and the
determinant of the incidence matrix of any $P$ game:
\begin{teo}\label{teo:42}
The incidence matrix $A$ of any $P$ game $G$ with $h$ types and
minimal winning quota $q$ is non singular and its determinant
is $|A|=(q)\cdot(-1)^{h+1}$.
\end{teo}

\begin{proof}
Let us denote by $a_i$, $i=1,\ldots,n$ the row vectors of $A$.
Recalling that the submatrix $M$ is upper triangular, we wish
to transform $A$ in a triangular matrix $T(A)$ with elements
$t_{ij}$ by adding to its last row a proper linear combination
$\sum_{i=1,\ldots,n-1}z_ia_i$ of the rows of $M$. This would
clearly imply that:
\begin{prop}\label{prop:41}
$$|T(A)|=|A|=t_{nn}$$
\end{prop}
We prove now:
\begin{prop}\label{prop:42}
The coefficients of such a linear combination are given by
$$z_i=(-1)^{t}\cdot w_i$$
if the coalition $i$ (of type $t$) is associated to a player
$i$ of type $t$.
\end{prop}
\end{proof}
To check that Prop. \ref{prop:42} holds we must verify that the
elements of the last row of $T$ satisfy:
\begin{subequations}\label{Eqn:42}
\begin{align}
t_{nj}=\sum_{i=1,\ldots,n-1}(-1)^{t}w_i\cdot
a_{ij}+a_{nj}=0\quad \text{for any $j=1,\ldots,n-1$}\label{Eqn42:first}\\
t_{nn}=\sum_{i=1,\ldots,n-1}(-1)^{t}w_i\cdot
a_{in}+a_{nn}=(q)\cdot(-1)^{h+1}\quad \text{for $j=n$}\label{Eqn42:second}
\end{align}
\end{subequations}

We now give the proof of equation \eqref{Eqn42:first}.
\begin{proof} As a consequence of Propositions
\ref{prop:321}, \ref{prop:322}, \ref{prop:323}, \ref{prop:325}
of section \ref{sect:32}, a null vector for the first $n-1$
components of the row vector $t_n$ is obtained if and only if
the following relations hold:
\begin{subequations}\label{Eqn:43}
\begin{align}
\sum_{0<s \hspace{0.1cm}\text{even}<t}x_s\cdot z_s+z_t+1=0\quad \text{for columns of type $t<h$ odd}\label{Eqn43:first}\\
\sum_{0<s \hspace{0.1cm} \text{odd}<t}x_s\cdot z_s+z_t+0=0\quad \text{for columns of type $t<h$ even}\label{Eqn43:second}
\end{align}
\end{subequations}
Then, starting from $t=1$, recursively:
$$
z_1=-1=-w_1 \quad\text{by \eqref{Eqn1:first}}
$$
for $t=2$:
$$
x_1\cdot z_1+z_2=0;\hspace{0.1cm} z_2=x_1\cdot w_1=w_2 \quad\text{by \eqref{Eqn1:second}}
$$
for $t=3$:
$$
x_2\cdot z_2+z_3+1=0;\hspace{0.1cm} z_3=-(1+ x_2\cdot w_2 )=-w_3 \quad\text{by \eqref{Eqn2:first}}
$$
for $t=4$:
$$
x_1\cdot z_1+ x_3\cdot z_3+z_4=0;\hspace{0.1cm} z_4=(x_1\cdot w_1+ x_3\cdot w_3 )=w_4 \quad\text{by \eqref{Eqn1:second}}
$$
and, by immediate induction on $t$, $z_t=-w_t$ for $t$ odd and
$z_t=w_t$ for $t$ even.
\end{proof}

We give now the proof of equation \eqref{Eqn42:second}.
\begin{proof} Suppose \eqref{Eqn42:first} holds
and consider at first an even $h$; then (by Proposition
\ref{prop:324}) $a_{in}=1$ for all coalitions $i$ of odd type
(odd $t$) and 0 for all coalitions of even type, while (by
Proposition \ref{prop:325}) $a_{nn}=0$, and we have
\begin{equation}\label{Eqn:boh}
t_{nn}=\sum_{i=1,\ldots,n-1}(-1)^{t}w_i\cdot a_{in}+a_{nn}=
-(x_1\cdot w_1+x_3\cdot w_3+\ldots+x_{h-1}\cdot w_{h-1})
\end{equation}

Formula \eqref{Eqn:boh} gives $t_{nn}$ as the opposite of the
sum of the weights of all odd players, which is the total
weight $q$ of the minimal winning coalition associated to the
top player.

Hence $t_{nn}=-q=(q)\cdot(-1)^{h+1}$.

On the other side, if $h$ is odd $a_{nn}=1$ and $a_{in}=1$ for
all even $t$ and we have:
\begin{equation}\label{Eqn:boh2}
t_{nn}=\sum_{i=1,\ldots,n-1}(-1)^t\cdot w_i\cdot a_{in}+a_{nn}=
=(x_2\cdot w_2+x_4\cdot w_4+\ldots+x_{h-1}\cdot w_{h-1})+1
\end{equation}
Then \eqref{Eqn:boh2} gives $t_{nn}$ as the sum of the
weights of all even players and one of the players with minimum
weight, which is the total weight $q$ of a minimal winning
coalition associated to one peone.

Hence $t_{nn}=q=(q)\cdot(-1)^{h+1}$ and the proof of Theorem
\ref{teo:42} has been completed.

\end{proof}

After that Theorem \ref{teo:41} follows immediately from the
chain:
$$
(q)\cdot(-1)^{h+1}=|A|=|A^T|=|\overline{A}|=(\overline{q})\cdot(-1)^{h+1}
$$

\section{Balanced lotteries and twin relationships in $P$ games}

\begin{defi}
A lottery on a $n$ person $P$ game $G$ is a probability
distribution on the minimal winning coalitions of $G$, i.e. a column vector $\mathbf{p}=(p_1,\ldots,p_j,\ldots,p_n)$ with $p_j$ the probability assigned to the minimal winning coalition $S_j$.
\end{defi}

\begin{defi}\label{def:52}
A balanced lottery on $G$ is a lottery
$\textbf{p}$ which assigns to all
players the same probability $\pi$ to be a member of the
minimal winning coalition selected by the lottery. Hence and
more formally a balanced lottery should satisfy:
$$
\textbf{p}^T\cdot A=\pi \cdot\1^T
$$
\end{defi}

The following results connect balanced lotteries and twin
relationships in $P$ games.

\begin{teo}\label{teo:51}
Let $G$ and $\overline{G}$ be a couple of twin games. There exists
just one balanced lottery $\textbf{p}$ on $G$; its
probabilities are given by the normalized individual weights of
the twin $\overline{G}$ according to the following rule:
\begin{subequations}\label{Eqn:51}
\begin{align}
p_j&=\overline{w}_{n-j}/\overline{w}(N) \quad\text{for $j=1,\ldots,n-1$}\label{Eqn:51first}\\
p_n&=\overline{w}_{n}/\overline{w}(N) \label{Eqn:51second}
\end{align}
\end{subequations}
\end{teo}
while
\begin{equation}\label{Eqn:51c}
 \pi=\overline{q}/\overline{w}(N)=q/w(N)
\end{equation}

\begin{proof}
By definition \ref{def:52} a balanced lottery $\mathbf{p}$ on $G$ must
satisfy
\begin{equation}\label{Eqn:52}
\mathbf{p}^T\cdot A=\pi\cdot\textbf{1}^T
\end{equation}
while by Prop. \ref{prop:331}:
\begin{equation}\label{Eqn:53}
    \overline{A}\overline{\w}=A^T\overline{\w}=\overline{q}\textbf{1}
\end{equation}

Exploiting non singularity and hence invertibility of $A$,
multiply \eqref{Eqn:52} and \eqref{Eqn:53}
respectively by $A^{-1}$ and $(A^T)^{-1}$ to obtain:
\begin{equation}\label{Eqn:52primo}
\mathbf{p}^T=\pi\1^T A^{-1}\quad\text{or}\quad\mathbf{p}=\pi (A^{-1})^{T}\textbf{1}
\end{equation}
\begin{equation}\label{Eqn:53primo}
    \overline{\w}=\overline{q}(A^T)^{-1}\1
\end{equation}
Being $(A^{-1})^T=(A^T)^{-1}$ it is:
\begin{equation}\label{Eqn:54}
    \pi^{-1}\mathbf{p}=\overline{q}^{-1}\overline{\w}
\end{equation}
so as:
\begin{equation}\label{Eqn:55}
    \mathbf{p}=\cfrac{\pi}{\overline{q}}\overline{\w}
\end{equation}

Premultiplying both sides of \eqref{Eqn:55} by $\1^{T}$
we obtain
\begin{equation}\label{Eqn:56}
    \1^{T}\mathbf{p}=1=\cfrac{\pi}{\overline{q}}\1^{T}\overline{\w}=\cfrac{\pi}{\overline{q}}\overline{w}(N)
\end{equation}
or
\begin{equation}\label{Eqn:56primo}
    \cfrac{\pi}{\overline{q}}=\cfrac{1}{\overline{w}(N)}
\end{equation}
and finally, by Theorem \ref{teo:41}
\begin{equation}\label{Eqn:57}
   \pi=\cfrac{\overline{q}}{\overline{w}(N)}=\cfrac{q}{w(N)}
\end{equation}
which proves \eqref{Eqn:51c} and
\begin{equation}\label{Eqn:58}
   \mathbf{p}=\cfrac{1}{\overline{w}(N)}\overline{\w}
\end{equation}
which gives the probability vector of the balanced lottery,
which is unique thanks to the non singularity of $A$.
\end{proof}

To understand the behaviour of the first $n-1$ components of
the vector $\mathbf{p}$, given in Theorem \ref{teo:51}, recall
that by Prop. \ref{prop:332} the ordering of coalitions and
players in $A^T$ (in $\overline{G}$) is, except for the
coalition associated to the top player, reversed respect to the
one in $A$ (in $G$); this implies that, for such coalitions,
the probability assigned to a coalition of type $t$ (in $G$) is
the (normalized) weight of a player of type $h-t$ in
$\overline{G}$.

We conclude this section with:
\begin{prop}\label{Prop:51}
If the division of payoffs among members of a minimal winning
coalition is proportional to their individual weights (in such
a coalition), the balanced lottery gives to each player in the
game $G$ an expected payoff proportional to her individual
weight in $G$ (and not in $\overline{G}$!).
\end{prop}
We wish to underline here that the rationality of the formation
of a minimal winning coalition, with division of payoff among
its members proportional to their individual weights, goes back
to the early stage of $n$ person game theory in coalitional
form; in particular, for $P$ games the set of $n$
``imputations'' generated by this logic are a stable set
solution \textit{a la} Von Neumann-Morgenstern of the $P$ game
$G$\footnote{As well known \cite{VNM47} a stable set solution
is a set $V$ of imputations such that: there is no dominance
among imputations of $V$, and any imputation not in $V$ is
dominated by at least one imputation of $V$.}.

The proof of the Prop. \ref{Prop:51} is immediate as the
expected reward $E(j)$ for any player $j$ is given, keeping
account of \eqref{Eqn:56primo}, by:
\begin{equation}\label{Eqn:59}
    E(j)=(w_j/q)\pi = (w_j/q)(q/w(N)) =(w_j/w(N))
\end{equation}

In this way both expected payoffs and actual payoff division
are proportional to the weights of the minimal homogeneous
representation of the game. We could say that ex ante and ex
post fairness are reconciled. We recall that this result is
similar to the one proposed by Montero \cite{Montero} and
Montero-Vidal Puga \cite{Montero2} as the outcome of a more
sophisticated model of bargaining.

\section{Some examples}\label{Sect:examples}

\paragraph{Ex. 1}
Let $G$ be the nine person $P$ game with minimal homogeneous
representation given by:
$$26;\text{ }1,1,1,3,4,4,11,11,15$$
The incidence matrix $A$ of $G$ is then:
$$A=\begin{bmatrix}
1&0&0&1&0&0&1&1&0\\
0&1&0&1&0&0&1&1&0\\
0&0&1&1&0&0&1&1&0\\
0&0&0&1&1&1&0&0&1\\
0&0&0&0&1&0&1&1&0\\
0&0&0&0&0&1&1&1&0\\
0&0&0&0&0&0&1&0&1\\
0&0&0&0&0&0&0&1&1\\
1&1&1&0&1&1&0&0&1\\
\end{bmatrix}$$
The transposition of $A$ gives the following incidence matrix
$\overline{A}$ of the game $\overline{G}$, the twin of $G$,
whose minimal homogeneous representation is
$26;\text{}1,1,2,2,5,7,7,7,19$,
$$\overline{A}=A^T=\begin{bmatrix}
1&0&0&0&0&0&0&0&1\\
0&1&0&0&0&0&0&0&1\\
0&0&1&0&0&0&0&0&1\\
1&1&1&1&0&0&0&0&0\\
0&0&0&1&1&0&0&0&1\\
0&0&0&1&0&1&0&0&1\\
1&1&1&0&1&1&1&0&0\\
1&1&1&0&1&1&0&1&0\\
0&0&0&1&0&0&1&1&1\\
\end{bmatrix}$$
Note that in this matrix the ordering of all players, but the
top, on columns and of the corresponding coalitions on rows is
reversed with respect to the standard one; for example, the
first block of three rows corresponds to the three coalitions
of type $t=4$ in the game $\overline{G}$; such coalitions are
formed by a vice top player of type 4 and by the top player.

If we wish to have $\overline{A}$ in the form coherent with the
standard ordering of players and coalitions we modify $A^{T}$
into the modified transposed $A^{\tau}$:
$$\overline{A}=A^{\tau}=\begin{bmatrix}
1&0&1&1&0&1&1&1&0\\
0&1&1&1&0&1&1&1&0\\
0&0&1&0&1&0&0&0&1\\
0&0&0&1&1&0&0&0&1\\
0&0&0&0&1&1&1&1&0\\
0&0&0&0&0&1&0&0&1\\
0&0&0&0&0&0&1&0&1\\
0&0&0&0&0&0&0&1&1\\
1&1&0&0&1&0&0&0&1\\
\end{bmatrix}$$

It is immediate to check that $G$ and $\overline{G}$ have the
same minimal winning quota $q=\overline{q}=26$.

The balanced lottery on $G$ is given, according to
\eqref{Eqn:51first} and \eqref{Eqn:51second} by the vector:
$$(7/51,7/51,7/51,5/51,2/51,2/51,1/51,1/51,19/51)$$ Then, for
example, the balanced lottery gives probability 7/51 to the
coalition formed by one of the peones, the player with weight 3
and both the last but top players.

Conversely, the balanced lottery on $\overline{G}$ is given by
$$(11/51,11/51,4/51,4/51,3/51,1/51,1/51,1/51,15/51)$$ Note the
extreme low probability (1/51) given in both games to the
coalitions formed by the top and by one of the last but top
players.

\paragraph{Ex. 2}

Let $G$ be the nine person $P$ game with minimal homogeneous
representation given by:
$$25;\text{ }1,1,1,3,4,7,7,7,18$$
The incidence matrix $A$ of $G$ is then:
$$A=\begin{bmatrix}
1&0&0&1&0&1&1&1&0\\
0&1&0&1&0&1&1&1&0\\
0&0&1&1&0&1&1&1&0\\
0&0&0&1&1&0&0&0&1\\
0&0&0&0&1&1&1&1&0\\
0&0&0&0&0&1&0&0&1\\
0&0&0&0&0&0&1&0&1\\
0&0&0&0&0&0&0&1&1\\
1&1&1&0&1&0&0&0&1\\
\end{bmatrix}$$

The transposition of $A$ gives the following incidence matrix
$\overline{A}$ of the game $\overline{G}$, the twin of $G$,
whose minimal homogeneous representation is still $25;\text{
}1,1,1,3,4,7,7,7,18$:
$$A^T=\begin{bmatrix}
1&0&0&0&0&0&0&0&1\\
0&1&0&0&0&0&0&0&1\\
0&0&1&0&0&0&0&0&1\\
1&1&1&1&0&0&0&0&0\\
0&0&0&1&1&0&0&0&1\\
1&1&1&0&1&1&0&0&0\\
1&1&1&0&1&0&1&0&0\\
1&1&1&0&1&0&0&1&0\\
0&0&0&1&0&1&1&1&1\\
\end{bmatrix}$$

Note that $G$ and $\overline{G}$ have the same minimal
homogeneous representation, which reveals that they are the
same game, even if, at first sight, it seems different from
$A$. Really the difference is not substantial as it comes
merely from the reversion of the order of non top players and
associated coalitions.

The identity between $G$ and $\overline{G}$ is recovered once
we compute $A^{\tau}$.
$$A^{\tau}=\begin{bmatrix}
1&0&0&1&0&1&1&1&0\\
0&1&0&1&0&1&1&1&0\\
0&0&1&1&0&1&1&1&0\\
0&0&0&1&1&0&0&0&1\\
0&0&0&0&1&1&1&1&0\\
0&0&0&0&0&1&0&0&1\\
0&0&0&0&0&0&1&0&1\\
0&0&0&0&0&0&0&1&1\\
1&1&1&0&1&0&0&0&1\\
\end{bmatrix}$$

It is immediately checked that the $A^{\tau}$ version of
$\overline{A}$ is equal to $A$, which confirms that
$\overline{G}=G$, as revealed also by the bilateral symmetry of
the representation vector $_f\x$ or by the coincidence of
$_f\x$ and $_f\overline{\x}$ both equal to (3,1,1,3).

To resume we could say that
\begin{prop}\label{Prop:61}
$G$ and $\overline{G}$ are coincident (are the same game) if
there is equality between: a) $A$ and the $A^{\tau}$ version of
$\overline{A}$ or b) between the free type representations
$_f\x$ and $_f\overline{\x}$ or c) between the minimal
homogeneous representations $(q,\w)$ and
$(\overline{q};\overline{\w})$.
\end{prop}

On the contrary, the formal equality between $A$ and the $A^T$
version of $\overline{A}$ is not a necessary condition for the
coincidence between $G$ and $\overline{G}$, except for games
$G$ with only two types $(h=2)$.

Concerning this point the following result holds:

\begin{prop}\label{Prop:62}
The incidence matrix $A$ of any $P$ game $G$ with $h=2$ types
of players satisfies $A=A^T=A^{\tau}$. Hence for such games
surely $G=\overline{G}$.
\end{prop}

The proof is immediate keeping account that no reversion of the
order of non top players (and associated coalitions) is
possible when $h-1=1$, while all but the last entries of the
last row and the last column are 1; the same information comes
from the bilateral symmetry of $_f\x$ which implies equality of
$_f\x$ and $_f\overline{\x}$.

Finally the balanced lotteries on both $G$ and $\overline{G}$
are given by :
$$7/25,7/25,7/25,4/25,3/25,1/25,1/25,1/25,18/25$$

More generally it is obvious that:
\begin{prop}\label{Prop:62}
A pair $G$ and $\overline{G}$ of identical twin games have the
same balanced lottery, while not identical twins $G$ and
$\overline{G}$ have the same minimal winning quota but
different balanced lotteries.
\end{prop}

\section{Conclusions}

In his smart treatment of Parsimonious games, going back to the
early stage of game theory development, Isbell introduced a
twin relationship, which at least at our knowledge, did not
find any further investigation in relevant literature.

Looking carefully at the properties of twin games and
signalling preliminarily that twins have free type
representation vectors which are each other symmetric, we show
in this paper that a) twin games have the same minimal winning
quota and b) any $P$ game has just one balanced lottery, given
by the (properly normalized and reordered) vector of individual
weights of its twin.

Our paper is purely theoretical and we do not discuss here any
application of our results; yet we feel confident that, keeping
account of the prominent role played by symmetry in the design
of the universe, applications of symmetric properties in hard
as well as in social sciences may be found so that this could
be a promising road for future research\footnote{There is a
huge literature concerning symmetry in hard sciences; let us
recall here some prominent sentences: ``Symmetry is one idea by
which man through the ages has tried to comprehend and create
order, beauty and perfection'' \cite[p. 5]{We52}; ``Symmetry
considerations dominate modern fundamental physics both in
quantum theory and in relativity'' \cite[p. ix
preface]{BrCa03};``Symmetry plays an essential role in
science'' \cite[editor foreword]{Ros96}.}.


\end{document}